\documentclass{article}[12pt]
\usepackage[a4paper, total={6in, 8in}]{geometry}

\usepackage{amsmath, amsthm, amssymb, amsfonts, mathtools, bbm}                                                                                                                                               
\usepackage{enumerate}
\usepackage{mathrsfs}

\newcommand{\norm}[1]{\lVert#1\rVert}

\newtheorem{theorem}{Theorem}
\newtheorem{lemma}{Lemma}
\newtheorem{proposition}{Proposition}
\newtheorem{corollary}{Corollary}

\theoremstyle{remark}
\newtheorem{remark}{Remark}

\title{Hörmander's Inequality and Point Evaluations in de Branges Space}
\author{Alex Bergman}
\date{\today}

\begin{document}

\maketitle

\begin{abstract}
    Let $f$ be an entire function of finite exponential type less than or equal to $\sigma$ which is bounded by $1$ on the real axis and satisfies $f(0) = 1$. Under these assumptions Hörmander showed that $f$ cannot decay faster than $\cos(\sigma x)$ on the interval $(-\pi/\sigma,\pi/\sigma)$. We extend this result to the setting of de Branges spaces with cosine replaced by the real part of the associated Hermite-Biehler function. We apply this result to study the point evaluation functional and associated extremal functions in de Branges spaces (equivalently in model spaces generated by meromorphic inner functions) generalizing some recent results of Brevig, Chirre, Ortega-Cerdà, and Seip.
\end{abstract}

\section{Introduction}

An entire function, $f : \mathbb{C} \to \mathbb{C}$ is said to be of finite exponential type $\sigma$ if
\begin{equation*}
    \limsup_{\lvert z \rvert \to \infty} \frac{\log \lvert f(z) \rvert}{\lvert z \rvert} = \sigma < \infty.   
\end{equation*}

In \cite{MR72210} Hörmander showed that if $f$ is a real entire function of exponential type at most $\sigma$ and $f \in L^{\infty}(\mathbb{R})$. Then at any point, $\xi$, where $f$ attains its maximum, $f(\xi) = \norm{f}_{\infty}$, it is bounded below by
\begin{equation}\label{eq:1}
    f(x) \geq \norm{f}_{\infty} \cos(\sigma (x-\xi)), \; \; -\frac{\pi}{\sigma} \leq x-\xi \leq \frac{\pi}{\sigma}.
\end{equation}
We find this inequality quite remarkable as it is a lower bound on the function. Essentially it says that a real entire function of exponential type attaining its maximum cannot decay faster than cosine until we hit a zero of sine. 


The condition that an entire $f$ is of most exponential type $\sigma$ and is bounded on the real line may be written as
\begin{equation*}
    \frac{f(z)}{e^{-i\sigma z}}, \frac{f^{\#}(z)}{e^{-i\sigma z}} \in H^{\infty}(\mathbb{C}_{+}),
\end{equation*}
where $f^{\#}(z) = \overline{f(\overline{z})}$, and $H^{\infty}(\mathbb{C}_{+})$ is the the set of bounded analytic function in the upper half-plane. The equivalence of these two definitions is well-known, but not entirely obvious, see \cite{MR2215727} for one approach to this problem. It is also a consequence of a theorem of Krein, see page 192 in \cite{MR1447633}. This definition is amenable to generalization in the following way. An entire, $E$, satisfying
\begin{equation*}
    \lvert E(\overline{z}) \rvert < \lvert E(z) \rvert, \; \; z \in \mathbb{C}_{+},
\end{equation*}
is called Hermite-Biehler and we write $E \in HB$. We define the space $\mathcal{H}^{\infty}(E)$ as the set of entire functions
\begin{equation*}
    \frac{f(z)}{E(z)}, \frac{f^{\#}(z)}{E(z)} \in H^{\infty}(\mathbb{C}_{+}),
\end{equation*}
with the norm 
\begin{equation*}
    \norm{f}_{\mathcal{H}^{\infty}(E)} = \norm{f/E}_{\infty} = \sup_{x\in \mathbb{R}} \left\lvert f(x)/E(x)\right\rvert.
\end{equation*}
The function $S_{\sigma}(z) = e^{-i\sigma z}$ is $HB$ for each $\sigma > 0$ and $\mathcal{H}^{\infty}(S_{\sigma})$ recovers the class of entire functions of exponential type at most $\sigma$ which are bounded on the real line, as expected. One can consider a whole scale of spaces by replacing $H^{\infty}(\mathbb{C}_{+})$ by the Hardy space of the upper half-plane $H^{p}(\mathbb{C}_{+})$, $0 < p < \infty$, in the above equation. In particular, the Hilbert space case $p = 2$ are the well studied de Branges spaces and appear in many contexts. De Branges spaces are the central object of the monograph \cite{MR0229011} and along with canonical systems the focus of \cite{https://doi.org/10.48550/arxiv.1408.6022, MR3890099}. Following the usual convention we call the spaces $\mathcal{H}^{p}(E)$ de Branges spaces as well.

Our goal is to generalize Hörmander's inequality \eqref{eq:1} to the setting of de Branges spaces. Let $E \in HB$. We may write
\begin{equation*}
    E(z) = A(z) + iB(z),
\end{equation*}
where $A = (E+E^{\#})/2$ and $B = (E-E^{\#})/2i$ are real entire functions (an entire function is called real if it is real on the real-axis). The functions $A$ and $B$ play the role of cosine and (negative) sine in this general setting. It is known that the real simple zeros of $A$ and $B$ are interlacing, see also Thereom \ref{thm:HB}. Replacing $E$ by $E_{\alpha} = e^{i\alpha} E$ produces another $HB$ function, but preserves the associated de Branges spaces, $\mathcal{H}^{p}(E_{\alpha}) = \mathcal{H}^{p}(E)$, $0 < p \leq \infty$. We shall write
\begin{equation*}
    E_{\alpha}(z) = A_{\alpha}(z) + iB_{\alpha}(z),
\end{equation*}
where now $A_{\alpha}$ and $B_{\alpha}$ are the real and imaginary parts of $E_{\alpha}$. Our main result is the following generalization of Hörmander's inequality to $\mathcal{H}^{\infty}(E)$.
\begin{theorem}\label{thm:thm1}
    Let $E \in HB$ have no real zeros and $f \in \mathcal{H}^{\infty}(E)$ be a real entire function. Suppose $f(\xi) = \lvert E(\xi) \rvert \norm{f/E}_{\infty}$, for some $\xi \in \mathbb{R}$. Let $\alpha \in \mathbb{R}$ be such that $E(\xi) = e^{-i\alpha}\lvert E(\xi) \rvert$. Denote by $b_{l}$ and $b_{r}$ the simple zeros of $B_{\alpha}$ to the left and right of $\xi$. Then
    \begin{equation*}
        f(x) \geq \norm{f/E}_{\infty} A_{\alpha}(x), \; \; b_{l} \leq x \leq b_{r}.
    \end{equation*}
\end{theorem}

The lower bound is sharp since $A_{\alpha}$ satisfies the hypotheses of the theorem. If $E(0) = 1$ and $f \in \mathcal{H}^{\infty}(E)$ is a real entire function, such that $f/E$ attains its maximum at $0$ we get a particularly simple version of the Theorem.

\begin{corollary}\label{cor:at_0}
    Let $E \in HB$ have no real zeros and $E(0) = 1$. If $f \in \mathcal{H}^{\infty}(E)$ is a real entire function and $f(0) = \norm{f/E}_{\infty}$ = 1. Then
    \begin{equation*}
        f(x) \geq A(x), \; \; b_{-1} \leq x \leq b_{1},
    \end{equation*}
    where $b_{-1}$ is the first negative zero of $B$ and $b_{1}$ is the first positive zero of $B$.
\end{corollary}

Note that Hörmander's inequality follows from the corollary and the fact that real translations are isometric on the class $PW_{\sigma}^{\infty} = \mathcal{H}^{\infty}(S_{\sigma})$ (this is, in general, not true for de Branges spaces). We also state a variant of the Theorem which is not dependent on the sign of $f$ at $\xi$.

\begin{corollary}\label{cor:sign_free}
    Let $E \in HB$ have no real zeros and $f \in \mathcal{H}^{\infty}(E)$ be a real entire function. Suppose $\lvert f (\xi) \rvert = \lvert E(\xi) \rvert \norm{f/E}_{\infty}$, for some $\xi \in \mathbb{R}$. Let $\alpha \in \mathbb{R}$ be such that $E(\xi) = e^{-i\alpha}\lvert E(\xi) \rvert$. Denote by $a_{l}$ and $a_{r}$ the simple zeros of $A_{\alpha}$ to the left and right of $\xi$. Then
    \begin{equation*}
        \lvert f(x) \rvert \geq \norm{f/E}_{\infty} A_{\alpha}(x), \; \; a_{l} \leq x \leq a_{r}.
    \end{equation*}
\end{corollary}

Hörmander's proof of \eqref{eq:1} is based on counting the number of zeros of the function $f(x)-\cos(\sigma x)$ (in the case $\norm{f}_{\infty} \leq 1$ and $f(0) = 1$). He achieves this by approximating the function with carefully chosen polynomials and applying Hurwitz Theorem. We shall prove our generalization by counting zeros of the function $f(x)-A_{\alpha}(x)$. However, in this general framework Hörmander's approximation scheme seems difficult to apply. Instead we shall use a variant of the Hermite-Biehler Theorem to control the number of zeros.

Before we begin with the proof of Theorem \ref{thm:thm1} we shall describe our applications of the result.

\subsection{Applications}

For $0<p\leq\infty$ let $PW_{\sigma}^{p}$ be the set of entire functions of exponential type at most $\sigma$ which are $p$-integrable on the real line
\begin{equation*}
    \norm{f}_{p}^{p} = \int_{-\infty}^{\infty} \lvert f(x) \rvert^{p}dx < \infty.
\end{equation*}
For $p = 2$ one recovers the usual Paley-Wiener space which is the Fourier image of functions supported on $[-\sigma, \sigma]$ (under a suitable normalization of the Fourier transform). Analogous results are true for other $p$, but must be interpreted in the sense of distributions if $p > 2$, see, for example page 172 in \cite{MR1447633}. In terms of de Branges spaces the function $S_{\sigma}(z) = e^{-i\sigma z}$ is $HB$ for each $\sigma > 0$ and $\mathcal{H}^{p}(S_{\sigma}) = PW_{\sigma}^{p}$. It is well-known that $PW_{\sigma}^{p} \subset PW_{\sigma}^{\infty}$, for all $0 < p < \infty$, see page 51 in \cite{MR1400006}. Thus the embedding operator
\begin{equation*}
    i_{p,\sigma} : PW_{\sigma}^{p} \to PW_{\sigma}^{\infty},
\end{equation*}
is bounded. Its norm appears in several problems in analysis. In \cite{MR1260445} Hörmander and Bernhardsson connect the norm of $i_{1, \pi}$ to a Bohr-type estimate for the Cauchy-Riemann operator in $\mathbb{R}^{2}$. More recently, in the work of Levin and Lubinsky, it has appeared as a scaling limit of certain Christoffel functions, see \cite{MR3317903, MR3691058}. In \cite{MR4162465} the norm of $i_{p,\sigma}$ has been connected to the study of Nikolski estimates for trigonometric polynomials.

To explain the known estimates we shall for definiteness consider $\sigma = \pi$ and $i_{p} = i_{p,\pi}$ in what follows. There has been a substantial amount of recent activity in estimating the norm of $i_{p}$. Until recently the best known bound was $\norm{i_{p}}^{p} \leq \lceil p/2 \rceil$, obtained using what is known as the power trick, see \cite{MR27332, MR114918, MR4417628}. The power trick will be explained later in the introduction. A refinement of the power trick allowed the authors of \cite{point_eval_in_PW} to obtain the current best known bound for small $p$,
\begin{equation}\label{eq:small_p_bound}
    \norm{i_{p}}^{p} < p/2, \; 0 < p <\infty.
\end{equation}
For large values of $p$ the best known bound is
\begin{equation}\label{eq:large_p_bound}
    \norm{i_{p}}^{p} \leq \sqrt{\frac{\pi p}{2}} + O(p^{-1/2}),
\end{equation}
and this is asymptotically sharp, see \cite{point_eval_in_PW}. For improvements for specific values of $p$ see also \cite{point_eval_in_PW, instanes2024optimizationproblempointevaluationpaleywiener, bondarenko2024hormanderbernhardssonsextremalfunction}. The case $p=1$ is particularly favorable and very precise estimates are presented in \cite{MR1260445}, see also \cite{bondarenko2024hormanderbernhardssonsextremalfunction}.

Since real translations are isometric in Paley-Wiener spaces the norm of the embedding operator, $i_{p,\sigma}$, is the same as the norm of the point evaluation functional at $0$, i.e. the smallest constant, such that
\begin{equation*}
    \lvert f(0) \rvert \leq C \norm{f}_{p}, \; f \in PW_{\sigma}^{p}.
\end{equation*}
Thus when we are faced with the task of generalizing the Paley-Wiener results to de Branges spaces we obtain two separate problems, the first being the study of the norm of the embedding operator (if the embedding is possible), and the second being the study of the point evaluation functional. Let us begin by discussing the embedding operator.

\subsubsection{Norm of the Embedding Operator}

We are interested in the embedding of $\mathcal{H}^{p}(E)$ into $\mathcal{H}^{\infty}(E)$. However, it is not always true that $\mathcal{H}^{p}(E)$ is contained inside $\mathcal{H}^{\infty}(E)$. We shall now present a theorem due to Dyakonov which essentially gives a complete characterization of when the embedding is possible. We shall present the solution in terms of the so called phase-function $\varphi$. The function
\begin{equation*}
    \Theta_{E}(z) = \frac{E(z)^{\#}}{E(z)},
\end{equation*}
is analytic in the upper half-plane and meromorphic in the complex plane. In addition, it is unimodular on the real-axis. Thus $\Theta_{E}$ is a so-called meromorphic inner-function and may be factored as
\begin{equation*}
    \Theta_{E}(z) = e^{i a z}S(z), z \in \mathbb{C}_{+},    
\end{equation*}
where $S$ is a Blaschke product and $a \geq 0$. A phase function is any smooth branch of the argument of $\Theta_{E}$ on the real-axis,
\begin{equation*}
    \Theta_{E}(x) = e^{i\varphi(x)}.
\end{equation*}
Phase functions are unique up to a constant factor of $2\pi$. Since our conditions will involve the derivative of $\varphi$ we shall slightly abuse notation and say the phase function instead of a phase function. We also see that in terms of the argument of $E$ we have $\varphi(x) = -2\arg(E(x))$. The phase function is an increasing function as can be seen from the relation
\begin{equation*}
    \varphi'(x) = -i \frac{\Theta_{E}'(x)}{\Theta_{E}(x)} = a + 2\sum_{n} \frac{y_{n}}{(x-x_{n})^{2}+y_{n}^{2}} > 0,
\end{equation*}
where $a$ and $z_{n} = x_{n}+iy_{n}$ are the parameters appearing in the canonical factorization of $\Theta_{E}$ as an exponential times a Blaschke product. Dyakonov has shown that for $1 < p < q \leq \infty$ the condition $\varphi' \in L^{\infty}(\mathbb{R})$ is necessary and sufficient for the embedding $\mathcal{H}^{p}(E) \subset \mathcal{H}^{q}(E)$, see \cite{MR1111913}. For $0 < p \leq 1$ the condition is still sufficient, but it is no longer necessary, see \cite{MR2231626}. We state this result as a theorem.

\begin{theorem}[Dyakonov \cite{MR1111913}]\label{thm:dyakonov}
    Suppose $\varphi' \in L^{\infty}$ and $0 < p < q \leq \infty$. Then $\mathcal{H}^{p}(E) \subset \mathcal{H}^{q}(E)$. Moreover, if $0 < p < \infty$, then
    \begin{equation*}
        \lim_{x \to \pm \infty} \left\lvert \frac{f(x)}{E(x)} \right\rvert^{p} = 0, \; \; f \in \mathcal{H}^{p}(E).
    \end{equation*}
\end{theorem}

\begin{remark}
    Let us mention that in \cite{MR2231626} the result is proved for model spaces. This does not cause any difficulties. The theorem can also be proved using Lemma 2 in \cite{MR1923406}.
\end{remark}

Since we are interested in studying the embedding for all $0 < p < \infty$ the condition $\varphi' \in L^{\infty}$ is necessary and sufficient for our purposes. Apart from this the condition $\varphi' \in L^{\infty}$ has appeared in many important works. Below we mention only a few. For example, it was used in \cite{MR2628803} to characterize Pólya sequences. In \cite{MR1923406, MR1111913} to characterize boundedness and compactness of the derivative operator on model spaces. In \cite{MR4722033} to study Carleson measures and oversampling in de Branges spaces. Similar, stronger, conditions also play an important role in Riesz bases for Paley-Wiener and de Branges spaces, see \cite{MR1923965, MR2264717}. In connection with this we also mention the interesting paper \cite{MR3631454}, and survey \cite{alexei_rupam_bounded_phase} where the authors study the problem of sequences, $\Lambda$, supporting a meromorphic inner function, $\Theta$, satisfying $\left\{x \in \mathbb{R} : \Theta(x) = 1 \right\} = \Lambda$ with $\varphi' \in L^{\infty}$.

In \cite{MR1111913} Dyakonov was only interested in the possibility of embedding, not estimating the norm. In fact, the implied constant obtained there blows up as $p \to \infty$. Suppose $\varphi' \in L^{\infty}$ and define $i_{p,E}$ to be the embedding operator
\begin{equation*}
    i_{p, E} : \mathcal{H}^{p}(E) \to \mathcal{H}^{\infty}(E).
\end{equation*}
Denote the norm of $i_{p,E}$ by $C(p,E) = \norm{i_{p,E}}$. Our estimate for the embedding operator is the following.

\begin{theorem}\label{thm:embedding_norm}
    Let $E \in HB$ have no real zeros. Suppose $\varphi' \in L^{\infty}$. Then
    \begin{equation*}
        C(p, E)^{p} \leq \norm{\varphi'}_{\infty}\frac{1}{2\sqrt{2\pi}}\left( \sqrt{p} + O(p^{-1/2})\right),
    \end{equation*}
    where the implicit constant does not depend on $\norm{\varphi'}_{\infty}$.
\end{theorem}
Let us make some comments on the theorem. From our proof we obtain the following non-asymptotic bound
\begin{equation}\label{eq:non_asymptotic}
    C(p, E)^{p} \leq \norm{\varphi'}_{\infty}2^{-1}\sqrt{\frac{p+1}{2\pi}},
\end{equation}
which is worse than \eqref{eq:small_p_bound} for small values of $p$. Also note that our bound is asymptotically correct as the $p$:th root of the right hand-side tends to $1$ as $p \to \infty$. In the Paley-Wiener case, $S_{\pi}(z) = e^{-i\pi z}$, $\norm{\varphi'}_{\infty} = 2\pi$ and we recover the best known asymptotic upper-bound
\begin{equation*}
    C(p, S_{\pi})^{p} \leq \sqrt{\frac{\pi p}{2}} + O(p^{-1/2}).
\end{equation*}
Let us also remark that the method used in \cite{point_eval_in_PW} to obtain \eqref{eq:small_p_bound}, which is a refinement of the power trick method used in \cite{MR27332, MR114918, MR4417628} do not apply in the setting of de Branges spaces. Indeed, the so-called power-trick is based on the following observation. If $p = kq$ where $k$ is a positive integer, then $g(z) = f^{k}(z/k) \in PW_{\pi}^{p}$, whenever $f \in PW_{\pi}^{q}$, and $\norm{g}_{p}^{p} = k\norm{f}_{q}^{q}$. This trick does not work, in general, in de Branges spaces. Instead, we shall base our proof on Theorem \ref{thm:thm1}. This method was also used in \cite{point_eval_in_PW} to obtain \eqref{eq:large_p_bound}.

\subsubsection{Point Evaluations}

In contrast to the embedding situation point evaluations are always bounded in de Branges spaces. For $0 < p \leq \infty$, $\xi \in \mathbb{R}$, and $E \in HB$ let $C(p,E,\xi)$ denote the best constant, such that
\begin{equation*}
    \lvert f(\xi)/E(\xi) \rvert \leq C(p,E,\xi)\norm{f/E}_{p}, \; f \in \mathcal{H}^{p}(E).
\end{equation*}
A normal families argument shows that extremal functions exist for all $p$ (a function is called extremal if the above inequality is an equality). For $p \geq 1$ standard arguments show that the extremal functions are unique, up to multiplication by a constant. Most of the basic properties of the extremal functions in the Paley-Wiener case from \cite{point_eval_in_PW} carry over without extra difficulty. In particular, an orthogonality relationship between the zeros of the extremal functions, which played a crucial role in \cite{point_eval_in_PW}, is essentially the same, see Proposition \ref{prop:variational_int}. However, finer properties of the extremal functions, such as, uniform separation of zeros fails in general, even in the case $p =2$. In fact, for $p=2$ uniform separation of zeros is equivalent to boundedness of the derivative of the phase function. Therefore, it seems natural to prove separation of zeros when $\varphi' \in L^{\infty}$. We have achieved this in the case $p \geq 1$.

\begin{theorem}\label{thm:sep_zeros}
        Let $E \in HB$ have no real zeros, $p \geq 1$, $\varphi' \in L^{\infty}$, and $f = f_{p,E,\xi}$ satisfy
        \begin{equation*}
            \lvert f(\xi)/E(\xi) \rvert = C(p,E,\xi)\norm{f/E}_{p}.
        \end{equation*}
        Then the zeros of $f$ are uniformly separated.
\end{theorem}

This generalizes Theorem 1.4. in \cite{point_eval_in_PW}, although there it was proven in the larger range $p \geq 1/2$. Our proof makes essential use of Theorem \ref{thm:thm1}.

\subsection*{Notation}
    Let $0 < p \leq \infty$. In this article $\norm{\cdot}_{p}$ always denotes the $L^{p}$ norm on the real-axis. We denote by $H^{p}$ the usual holomorphic Hardy space of the upper half-plane. We identify functions in $H^{p}$ with their non-tangential boundary values on $\mathbb{R}$. For an entire function $f$ we define $f^{\#}(z) = \overline{f(\overline{z})}$.

\subsection*{Acknowledgment}

The author wishes to thank Joaquim Ortega Cerdà and Kristian Seip for helpful comments and suggestions for further research.

\section{Proof of Theorem \ref{thm:thm1}}

We denote by $\overline{HB}$ the class of entire function satisfying $\lvert E(\overline{z}) \rvert \leq \lvert E(z) \rvert$, $z \in \mathbb{C}_{+}$. Thus the difference between $HB$ and $\overline{HB}$ is that in the latter we allow the degenerate case when $E$ is a scalar multiple of a real entire function. This will make the statements of certain results simpler. We begin with a variant of the well-known Hermite-Biehler Theorem. We shall only need a simple part of the full Theorem and therefore we include a proof for completeness.

\begin{theorem}\label{thm:HB}
    Let $E \in HB$, $E = A + i B$. Then there exists a real entire function $S$ with only real zeros, such that $A = SA_{0}$, $B = SB_{0}$. The real zeros of $A_{0}$ and $B_{0}$ are simple and interlacing. Moreover,
    \begin{equation}\label{eq:derivative}
        A'(x)B(x) - A(x)B'(x) \geq 0, \; x \in \mathbb{R},
    \end{equation}
    with strict inequality at all points, $x$, such that $E(x) \neq 0$.
    \begin{proof}
        It is clear that we may define a real entire function $S$ as in the statement of the Theorem. It remains to prove that the real zeros of $A_{0}$ and $B_{0}$ are simple, interlacing, and equation \eqref{eq:derivative} holds. We begin by proving the identity in equation \eqref{eq:derivative}. The function $\Theta_{E} = E^{\#}/E$ is bounded by $1$ in the upper half-plane. Thus
        \begin{equation*}
            z \mapsto i \frac{1+\Theta_{E}}{1-\Theta_{E}},
        \end{equation*}
        maps $\mathbb{C}_{+}$ to itself and also it is real on the real line, except on the exceptional set $\left\{ \Theta_{E}(x) = 1 \right\}$. Basic geometric considerations show that the normal derivative of the imaginary part is nonpositive and hence the $y$-derivative of the imaginary part of the function is nonnegative on the real-axis. By the Cauchy-Riemann equations it follows that the $x$-derivative of the real part is nonpositive on the real-axis. A computation gives
        \begin{equation*}
            i \frac{1+\Theta_{E}}{1-\Theta_{E}} = A/B.
        \end{equation*}
        Thus we obtain
        \begin{equation}
            0 \leq \left(\frac{A(x)}{B(x)}\right)'= \frac{B(x)A'(x)-A(x)B'(x)}{B(x)^{2}}, \; x \in \mathbb{R}.
        \end{equation}
        In fact, since the imaginary part is a positive harmonic function, which is zero on the real line, it follows from the usual Hopf lemma for elliptic PDE that we have strict inequality, except possibly on the exceptional set $\left\{ \Theta_{E}(x) = 1 \right\} = \left\{ B(x) = 0 \right\}$. For a proof of the Hopf lemma see, for example, Chapter 6 in \cite{MR2597943} or Chapter 2 in \cite{MR2450237} for a proof applicable to harmonic functions. Analogous considerations for $-B/A$ give that the inequality is strict except possibly on the exceptional set $\left\{ A(x) = 0 \right\}$. Thus whenever $E(x) \neq 0$ we have the strict inequality
        \begin{equation}\label{eq:interlacing_ineq}
            B(x)A'(x)-A(x)B'(x) > 0.
        \end{equation}
        It remains to prove the interlacing property. We shall deduce it from equation \eqref{eq:derivative}. By replacing $E$ by $E/S$ if necessary we may suppose $A$ and $B$ have no common zeros. Suppose $A$ has a zero at $x_{0}$ of order larger than or equal to $2$, then we have equality in \eqref{eq:derivative} and hence $B(x_{0}) = 0$, which is impossible. Thus $A$ has only simple zeros. Analogous considerations give that $B$ has only simple zeros as well. Now let $x_{0}$ be a real simple zero of $A$, then
        \begin{equation*}
            A'(x_{0})B(x_{0}) > 0.
        \end{equation*}
        So $A'(x_{0})$ and $B(x_{0})$ have the same sign. Suppose $A$ has another simple zero to the right of $x_{0}$ before $B$ has a zero to the right of $x_{0}$. Then $B$ has the same sign, but $A'$ has the opposite sign, contradicting \eqref{eq:interlacing_ineq}. The same argument applies to the left and to zeros of $B$ so the result is proved.
    \end{proof}
\end{theorem}

Our main technical tool is the following Lemma.

\begin{lemma}\label{lemma:main_lemma}
    Let $E$ and $f$ be entire functions. Then
    \begin{equation*}
        f(z)-\lambda E(z) \in \overline{HB},
    \end{equation*}
    for all complex $\lvert \lambda \rvert \geq 1$, if $E \in \overline{HB}$, $f \in \mathcal{H}^{\infty}(E)$, and $\norm{f/E}_{\infty} \leq 1$.
    \begin{proof}
        Let $E \in \overline{HB}$, $f \in \mathcal{H}^{\infty}(E)$ and $\norm{f/E}_{\infty} \leq 1$. Let $\lambda \in \mathbb{C}$ and $\lvert \lambda \rvert \geq 1$. We must show that
        \begin{equation*}
            f(z)-\lambda E(z) \in \overline{HB}.
        \end{equation*}
        Assume that $\lvert \lambda \rvert > 1$. Then for $z \in \mathbb{C}_{+}$ we have by the $HB$ property of $E$ and that $\norm{f/E}_{\infty} \leq 1$
        \begin{equation*}
            \left\lvert \frac{f^{\#}(z)-\overline{\lambda}E^{\#}(z)}{f(z)- \lambda E(z)} \right\rvert \leq \left\lvert \frac{f^{\#}(z)/E(z)-\overline{\lambda}E^{\#}(z)/E(z)}{f(z)/E(z)- \lambda} \right\rvert \leq \frac{1 + \lvert \lambda \rvert}{\lvert \lambda\rvert-1}.
        \end{equation*}
        Thus the holomorphic function inside the absolute value sign on the left is bounded in the upper half-plane and hence it achieves its supremum norm on the real axis. Since it is bounded by $1$ on the real-axis it follows that $f-\lambda E \in \overline{HB}$ in this case. For $\lvert \lambda \rvert = 1$ the result follows by taking a limit.

    \end{proof}
\end{lemma}

\begin{remark}
    In fact, the converse of the previous lemma is also true, but we shall not need it.
\end{remark}

For the proof of Theorem \ref{thm:thm1} we may assume $\norm{f/E}_{\infty} \leq 1$. Let $\varphi_{\alpha}$ be the function
\begin{equation*}
    \varphi_{\alpha}(z) = f(z) - e^{i\alpha}E(z).
\end{equation*}
It follows from Lemma \ref{lemma:main_lemma} that $\varphi_{\alpha} \in \overline{HB}$. If we are in the degenerate case where $\varphi_{\alpha}$ is a scalar multiple of a real entire function for some $\alpha$ then the theorem is straightforward and hence we assume $\varphi_{\alpha} \in HB$ from now on. Write
\begin{equation*}
    \varphi_{\alpha}(z) = \Omega_{\alpha}(z) + i \Gamma_{\alpha}(z),    
\end{equation*}
where $\Omega_{\alpha}$ and $\Gamma_{\alpha}$ are the real and imaginary parts of $\varphi_{\alpha}$ respectively. Since $f$ is real entire we have
\begin{equation*}
    \Omega_{\alpha}(z) = \frac{1}{2} \left(f(z)-e^{i\alpha}E(z) + \left(f(z)-e^{i\alpha}E(z)\right)^{\#}\right) = f(z)-A_{\alpha}(z), 
\end{equation*}
and
\begin{equation*}
    \Gamma_{\alpha}(z) = \frac{1}{2i} \left(f(z)-e^{i\alpha}E(z) - \left(f(z)-e^{i\alpha}E(z)\right)^{\#}\right) = -B_{\alpha}(z), 
\end{equation*}
Suppose $\xi \in \mathbb{R}$ is such that $f(\xi) = e^{i\alpha}E(\xi)$. Our goal is to show that $\Omega_{\alpha}(x) \geq 0$ on the interval $b_{l} \leq \xi \leq b_{r}$, where $b_{l}$ and $b_{r}$ are the simple zeros of $B_{\alpha}$ to the left and right of $\xi$ respectively. We will do this in two steps. First we prove that the positivity propagates to the whole desired interval. That is, if $\Omega_{\alpha}(x) > 0$ for some $x \in (b_{l}, b_{r})$ then necessarily $\Omega_{\alpha}(x) \geq 0$ for all $x \in [b_{l}, b_{r}]$. Then we show that $\Omega_{\alpha}(x) > 0$ in some punctured neighbourhood of $\xi$. 

\subsection{Propagation of Positivity}

Suppose $\xi \in \mathbb{R}$ is such that $f(\xi) = e^{i\alpha}E(\xi)$. Then $\varphi_{\alpha}(\xi)=0$ and hence both $\Omega_{\alpha}(\xi) = 0$ and $\Gamma_{\alpha}(\xi) = 0$. Also since both $f/\lvert E_{\alpha} \rvert$ and $A_{\alpha}/\lvert E_{\alpha} \rvert$ have extreme points at $\xi$ it will follow that $\Omega_{\alpha}$ has a zero of order at least $2$ at $\xi$. We show that $\varphi_{\alpha}$ has only a simple zero at $\xi$ and hence one of the zeros of $\Omega_{\alpha}$ at $\xi$ is a simple interlacing zero.

\begin{lemma}\label{lemma:simple_zero}
    The real zeros of $\varphi_{\alpha}$ are all simple.
    \begin{proof}
        Suppose, seeking a contradiction, that $\varphi_{\alpha}$ has a zero of order $p \geq 2$ at $\eta \in \mathbb{R}$. Then
        \begin{equation*}
            (z-\eta)^{p}h(z) = f(z)-e^{i\alpha}E(z), \; \; h(\eta) \neq 0.
        \end{equation*}
        Since $\eta$ is real $E(\eta) \neq 0$. Write $h(\eta)/E(\eta) = e^{i\beta} \lvert h(\eta)/E(\eta) \rvert$. Choose a sequence $\left\{ z_{n} \right\} \subset \mathbb{C}_{+}$, $z_{n} \to \eta$, such that
        \begin{equation*}
            (z_{n}-\eta)^{p} = e^{i\alpha}e^{-i\beta}\lvert z_{n} - \eta \rvert^{p}.
        \end{equation*}
        This is possible since $p \geq 2$. Then $h(z_{n})/E(z_{n}) = e^{i(\beta + \delta_{n})}\lvert h(z_{n})/E(z_{n}) \rvert$, with $\delta_{n} \to 0$.
        \begin{equation*}
        \begin{split}
            \left\lvert \frac{f(z_{n})}{E(z_{n})} \right\rvert^{2} = \left\lvert(z_{n}-\eta)^{p}\frac{h(z_{n})}{E(z_{n})}+e^{i\alpha} \right\rvert^{2} \\ = 1 + \lvert z_{n} - \eta \rvert^{2p}\left\lvert \frac{h(z_{n})}{E(z_{n})} \right\rvert^{2} + 2 \lvert z_{n} - \eta \rvert^{p} \left\lvert \frac{h(z_{n})}{E(z_{n})} \right\rvert \cos(\delta_{n}) > 1,
        \end{split}
        \end{equation*}
        for $\delta_{n}$ sufficiently small. This is a contradiction, hence $\eta$ is a zero of order $1$.
    \end{proof}
\end{lemma}

In particular, this means that the zero of $\varphi_{\alpha}$ at $\xi$ is order $1$.

\begin{lemma}\label{lemma:order_of_zero}
    The zero of $\Omega_{\alpha}$ at any real zero of $\varphi_{\alpha}$ is of order $2$.
    \begin{proof}
        Let $\eta \in \mathbb{R}$ be a zero of $\varphi_{\alpha}$. First we note that it is of most order $2$ since if it were of higher order then $\varphi_{\alpha}$ would have a zero of order $\geq 2$, which is impossible by the previous Lemma. It remains to prove that $\Omega_{\alpha}'(\eta) = 0$. By definition
        \begin{equation*}
            \Omega_{\alpha}'(\eta) = f'(\eta) - A_{\alpha}'(\eta).
        \end{equation*}
        Thus to prove the Lemma it will suffice to show that $f'(\eta) = A_{\alpha}'(\eta)$. Since the smooth functions $f(x)/\lvert E(x) \rvert$ $A_{\alpha}(x)/\lvert E(x) \rvert$ have extreme points at $\eta$ it follows by differentiating them that
        \begin{equation*}
            \frac{f'(\eta)}{f(\eta)} = \frac{\partial_{x} \lvert E \rvert (\eta)}{\lvert E(\eta) \rvert} \; \text{ and }\; \frac{A_{\alpha}'(\eta)}{A_{\alpha}(\eta)} = \frac{\partial_{x} \lvert E \rvert (\eta)}{\lvert E(\eta) \rvert}.
        \end{equation*}
        Since $f(\eta) = A_{\alpha}(\eta) = e^{i\alpha}E(\eta)$ it follows that $f'(\eta) = A_{\alpha}'(\eta)$ as required.
    \end{proof}
\end{lemma}
Thus $\Omega_{\alpha}$ cannot have any simple zeros in the interval $[b_{l}, b_{r}]$, where $b_{l}$ and $b_{r}$ are the simple zeros of $\Gamma_{\alpha} = -B_{\alpha}$ to the right and left of $\xi$ respectively. Since any zero of $\Omega_{\alpha}$ in $[b_{l}, b_{r}]$ is a zero of order $2$ the sign of $\Omega_{\alpha}$ is constant in $[b_{l}, b_{r}]$. Thus Theorem \ref{thm:thm1} will follow if we show that
\begin{equation*}
    \Omega_{\alpha}(x) > 0, \; \text{ for some } b_{l} < x <  b_{r}.
\end{equation*}

We do this in the next section.

\subsection{Positivity in a Small Neighbourhood of $\xi$}

We now show that $\Omega_{\alpha}(x) > 0$ in some punctured neighbourhood of $\xi$. We isolate this fact as a Lemma.

\begin{lemma}\label{lemma:positivity}
    There exists a punctured neighbourhood of $\xi$, such that $\Omega_{\alpha}(x) > 0$.
    \begin{proof}
        By Lemma \ref{lemma:simple_zero} $\varphi_{\alpha}$ has a simple zero at $\xi$. By Lemma \ref{lemma:order_of_zero} $\Omega_{\alpha}$ has a double zero at $\xi$, hence we find that
        \begin{equation*}
            i\Gamma_{\alpha}'(\xi) = \varphi_{\alpha}'(\xi) = e^{i\gamma}\lvert \varphi'(\xi) \rvert.
        \end{equation*}
        Thus either $\gamma = \pi/2$ or $\gamma = -\pi/2$ $\mod 2\pi$. We claim that we have the positive sign. Indeed, write
        \begin{equation*}
            \frac{f(z)}{E(z)} = e^{i\alpha} + (z-\xi)\frac{\varphi_{\alpha}(z)}{(z-\xi)E(z)},
        \end{equation*}
        where $\varphi_{\alpha}'(\xi) \neq 0$. Let $\varphi'_{\alpha}(\xi)= e^{i\gamma} \lvert \varphi'_{\alpha}(\xi) \rvert$. Recalling that $E(\xi) = e^{-i\alpha}\lvert E(\xi) \rvert$ we obtain
        \begin{equation*}
            \frac{f(z)}{E(z)} = e^{i\alpha}\left(1 + (z-\xi)e^{i\gamma}e^{i\delta(z)}\frac{\lvert \varphi_{\alpha}(z) \rvert}{\lvert (z-\xi) \rvert \lvert E(z) \rvert} \right),
        \end{equation*}
        with $\delta(z)$ real and tending to $0$ as $z \to \xi$. As in the proof of Lemma \ref{lemma:simple_zero} we see that
        \begin{equation*}
            0 < \gamma < \pi \mod{2\pi},
        \end{equation*}
        since otherwise we would obtain $\lvert f(z) / E(z) \rvert > 1$ for some $z$ in the upper half-plane. Thus indeed, $e^{i\gamma} = i$ as claimed. Since $\psi_{\alpha}(z) = (z-\xi)^{-1}\varphi_{\alpha}(z) \in HB$ and $\psi_{\alpha}(\xi) \neq 0$ we obtain from Theorem \ref{thm:HB}
        \begin{equation*}
            \Omega_{\alpha}''(\xi)\Gamma_{\alpha}'(\xi) > 0.
        \end{equation*}
        We have seen that $\Gamma_{\alpha}'(\xi) > 0$, hence
        \begin{equation*}
            \Omega_{\alpha}''(\xi) > 0.
        \end{equation*}
        It follows that $\Omega_{\alpha}(x) > 0$ in some punctured neighbourhood of $\xi$.
    \end{proof}
\end{lemma}

Theorem \ref{thm:thm1} now follows by combining Lemmas \ref{lemma:simple_zero} and \ref{lemma:positivity}. Corollaries \ref{cor:at_0} and \ref{cor:sign_free} follow immediately from the theorem.

\section{Proof of Theorem \ref{thm:embedding_norm}}

In this section we prove Theorem \ref{thm:embedding_norm}. Recall that we denote by $C(p,E)$ the norm of the operator
\begin{equation*}
    i_{p, E} : \mathcal{H}^{p}(E) \to \mathcal{H}^{\infty}(E).
\end{equation*}
We begin by reducing the problem to real entire functions.

\begin{lemma}\label{lemma:reduction_to_real}
    Let $0 < p < \infty$. Suppose there exists $C > 0$, such that $\norm{f/E}_{\infty} \leq C \norm{f/E}_{p}$ for all real entire functions $f \in \mathcal{H}^{p}(E)$. Then $C(p, E) \leq C$.
    \begin{proof}
        Let $f \in \mathcal{H}^{p}(E)$. Then $f/E$ decays to zero along the real-axis, by Theorem \ref{thm:dyakonov}, and hence there exists $\xi \in \mathbb{R}$, such that
        \begin{equation*}
            \lvert f(\xi)/E(\xi) \rvert = \norm{f/E}_{\infty}.
        \end{equation*}
        Let $\sigma \in \mathbb{R}$ be such that
        \begin{equation*}
            f(\xi)/E(\xi) = e^{i\sigma}\overline{f(\xi)}/E(\xi)
        \end{equation*}
        Let $g \in \mathcal{H}^{p}(E)$ be the real entire function
        \begin{equation*}
            g(z) = \frac{e^{-i\sigma/2}f(z) + e^{i\sigma/2}f^{\#}(z)}{2}.
        \end{equation*}
        It is clear that $\norm{g/E}_{\infty} \leq \norm{f/E}_{\infty}$ and
        \begin{equation*}
            \lvert g(\xi)/E(\xi) \rvert= \left\lvert \frac{e^{-i\sigma/2}f(\xi)/E(\xi) + e^{i\sigma/2}\overline{f(\xi)}/E(\xi)}{2} \right\rvert = \norm{f/E}_{\infty}.
        \end{equation*}
        Since $g$ is real entire it follows by assumption that
        \begin{equation*}
            \norm{f/E}_{\infty} = \norm{g/E}_{\infty} \leq C\norm{g/E}_{p} \leq C\norm{f/E}_{p}.
        \end{equation*}
        Thus $C(p,E) \leq C$ as claimed.
    \end{proof}
\end{lemma}

We are now ready for the main result of this section.

\begin{theorem}\label{thm:est_kp}
    Let $0 < p < \infty$. Then
    \begin{equation*}
        C(p) \leq \frac{\norm{\varphi'}_{\infty}^{1/p}}{2^{1/p}K(p)},
    \end{equation*}
    where
    \begin{equation*}
        K(p) = \left( \int_{-\frac{\pi}{2}}^{\frac{\pi}{2}} \lvert \cos(x) \rvert^{p}dx \right)^{1/p}.
    \end{equation*}
\end{theorem}
Before we begin the proof let us remark that estimating $K(p)$ will give us Theorem \ref{thm:embedding_norm}.
\begin{proof}
    By Lemma \ref{lemma:reduction_to_real} it suffices to consider real entire functions $f \in \mathcal{H}^{p}(E)$. By Theorem \ref{thm:dyakonov} there exists $\xi \in \mathbb{R}$, such that
    \begin{equation*}
        \lvert f(\xi) \rvert = \lvert E(\xi) \rvert \norm{f/E}_{\infty}.
    \end{equation*}
    Let $\alpha$ be such that $e^{i\alpha}E(\xi) = \lvert E(\xi) \rvert$, then by Corollary \ref{cor:sign_free}
    \begin{equation*}
        \lvert f(x) \rvert \geq \norm{f/E}A_{\alpha}(x), \; \; a_{l} \leq x \leq a_{r},
    \end{equation*}
    where $a_{l}$ and $a_{r}$ are the real simple zeros of $A_{\alpha}$ to the left and right of $\xi$, respectively. Thus we obtain
    \begin{equation*}
        \int_{\mathbb{R}} \frac{\lvert f(x) \rvert^{p}}{\lvert E(x) \rvert^{p}}dx \geq \int_{a_{l}}^{a_{r}} \frac{\lvert f(x) \rvert^{p}}{\lvert E(x) \rvert^{p}}dx \geq \norm{f/E}_{\infty}^{p} \int_{a_{l}}^{a_{r}} \frac{\lvert A_{\alpha}(x) \rvert^{p}}{\lvert E(x) \rvert^{p}}dx.
    \end{equation*}
    Thus to obtain estimates for $C(p)$ it will suffice to estimate (from below)
    \begin{equation*}
        I(\alpha, p) = \int_{a_{l}}^{a_{r}} \frac{\lvert A_{\alpha}(x) \rvert^{p}}{\lvert E(x) \rvert^{p}}dx,
    \end{equation*}
    where $a_{l}$ and $a_{r}$ are any two consecutive simple zeros of $A_{\alpha}$. Recalling that $A_{\alpha}$ is the real part of $e^{i\alpha}E$ we obtain
    \begin{equation*}
        I(\alpha, p) = 2^{-p} \int_{a_{l}}^{a_{r}} \lvert 1 + e^{-2i\alpha}\Theta_{E}(x) \rvert^{p}dx = 2^{-p/2} \int_{a_{l}}^{a_{r}} \lvert 1 + \text{Re}(e^{-2i\alpha}\Theta_{E}(x)) \rvert^{p/2}dx.
    \end{equation*}
    Using that $\Theta_{E}(x) = e^{i\varphi(x)}$ gives
    \begin{equation*}
        \int_{a_{l}}^{a_{r}} \lvert 1 + 2\text{Re}(e^{-2i\alpha}\Theta_{E}(x)) \rvert^{p/2}dx = \int_{a_{l}}^{a_{r}} \lvert 1 + \cos(\varphi(x)-2\alpha) \rvert^{p/2}dx.
    \end{equation*}
    Now we perform a change of variable $\varphi(x)-2\alpha = t$. Since $A_{\alpha}$ vanishes at $a_{l}$ and $a_{r}$ it follows that $e^{i\alpha}E$ is purely imaginary at those points, thus
    \begin{equation*}
        \alpha + \arg E(a_{j}) = \pm \frac{\pi}{2}, \; \mod 2\pi, \; j = l, r.
    \end{equation*}
    Using that $\varphi = -2\arg(E)$ we obtain
    \begin{equation*}
        \varphi(a_{j}) - 2\alpha = \mp \pi, \; \mod 4\pi, \; j = l, r.
    \end{equation*}
    Also, since the $a_{l}$ and $a_{r}$ are consecutive it follows that one has the positive sign and the other the negative sign. In fact, since $A_{\alpha}$ has a maximum at $a_{l} < \xi < a_{r}$ it follows that $\arg(e^{i\alpha}E(a_{l})) = \pi/2$ and $\arg(e^{i\alpha}E(a_{r})) = -\pi/2$. Collecting all this gives
    \begin{equation*}
        \int_{a_{l}}^{a_{r}} \lvert 1 + \cos(\varphi(x)-2\alpha) \rvert^{p/2}dx = \int_{-\pi}^{\pi} \lvert 1 + \cos(t) \rvert^{p/2}\frac{dt}{\varphi'(x)} \geq \frac{1}{\norm{\varphi'}_{\infty}} \int_{-\pi}^{\pi} \lvert 1 + \cos(t) \rvert^{p/2}dt.
    \end{equation*}
    Thus we obtain
    \begin{equation*}
        I(\alpha,p) \geq \frac{1}{\norm{\varphi'}_{\infty}}2^{-p/2} \int_{-\pi}^{\pi} \lvert 1 + \cos(t) \rvert^{p/2}dt = \frac{2}{\norm{\varphi'}_{\infty}} \int_{-\frac{\pi}{2}}^{\frac{\pi}{2}} \lvert \cos(x) \rvert^{p}dx = \frac{2K(p)^{p}}{\norm{\varphi'}_{\infty}}.
    \end{equation*}
    We see that the bound is independent of $\alpha$ and the points $a_{l}$, $a_{r}$. Hence,
    \begin{equation*}
        C(p) \leq \frac{\norm{\varphi'}_{\infty}^{1/p}}{2^{1/p}K(p)}.
    \end{equation*}
    The result is proved.
\end{proof}

To prove Theorem \ref{thm:embedding_norm} it remains to estimate $K(p)$. A change of variable gives
\begin{equation*}
    K(p)^{p} = \int_{0}^{1} x^{-1/2}(1-x)^{(p-1)/2}dx = B(1/2,(p+1)/2) = \sqrt{\pi}\frac{\Gamma(\frac{p+1}{2})}{\Gamma(\frac{p+2}{2})},
\end{equation*}
where $B$ is the Euler Beta function and $\Gamma$ is the usual gamma function, see Appendix A in \cite{MR3243734}. Standard asymptotic estimates for the gamma function now give
\begin{equation*}
    \frac{1}{K(p)^{p}} = \frac{1}{\sqrt{\pi}}\frac{\Gamma(\frac{p+2}{2})}{\Gamma(\frac{p+1}{2})} = \sqrt{\frac{p}{2\pi}} + O(p^{-1/2}), \; p \to \infty.
\end{equation*}
Theorem \ref{thm:embedding_norm} follows from the above asymptotics and Theorem \ref{thm:est_kp}. The non-asymptotic estimate for $C(p,E)$ in Equation \eqref{eq:non_asymptotic} can be obtained using the inequality
\begin{equation*}
    \Gamma(x+1/2) \leq x^{1/2}\Gamma(x), \; x > 0,
\end{equation*}
with $x = (p+1)/2$. A proof of this inequality may be found in \cite{MR29448}.

\section{The Extremal Functions}

    We wish to study the extremal functions with respect to the embedding $\mathcal{H}^{p}(E) \subset \mathcal{H}^{\infty}(E)$, that is functions satisfying
    \begin{equation*}
        \norm{f/E}_{\infty} = C(p,E) \norm{f/E}_{p}.
    \end{equation*}
    In this section we consider only $E \in HB$ function without real zeros, this is not really any loss of generality. In this case the space $\mathcal{H}^{p}(E)$ has no common zeros. We now highlight a key difference between Paley-Wiener spaces and generic de Branges spaces. There is no guarantee that extremal functions exist, even for $p=2$. Indeed, we have the following result.
    \begin{proposition}
        Let $E \in HB$ and $\varphi' \in L^{\infty}$. Then there exists a function $f \in \mathcal{H}^{2}(E)$ satisfying
        \begin{equation*}
            \norm{f/E}_{\infty} = C(2,E) \norm{f/E}_{2},
        \end{equation*}
        if and only if $\varphi'$ attains its supremum at a finite point.
        \begin{proof}
            For fixed $\xi \in \mathbb{R}$ the kernel function
            \begin{equation*}
                K_{\xi}(z) = \frac{E(z)\overline{E(\xi)}-E(z)^{\#}\overline{E(\xi)^{\#}}}{2\pi i(\xi-z)},
            \end{equation*}
            is extremal for the continuous linear functional
            \begin{equation*}
                \ell_{\xi}(f) = f(\xi)/\lvert E(\xi) \rvert, \; f \in \mathcal{H}^{2}(E).
            \end{equation*}
            Moreover,
            \begin{equation*}
                K_{\xi}(\xi) = \pi^{-1}\text{Im}(\overline{E'(\xi)}E(\xi)) = \frac{1}{2\pi} \lvert E(\xi) \rvert^{2} \varphi'(\xi).
            \end{equation*}
            Thus
            \begin{equation*}
                \frac{K_{\xi}(\xi)}{\lvert E(\xi) \rvert} = \frac{\norm{K_{\xi}/E}_{2}}{\lvert E(\xi) \rvert}\norm{K_{\xi}/E}_{2} = \frac{1}{\sqrt{2\pi}}\varphi'(\xi)^{1/2}\norm{K_{\xi}/E}_{2}.
            \end{equation*}
            The result now follows since
            \begin{equation*}
                C(2,E) = \sup_{x\in \mathbb{R}} \frac{1}{\sqrt{2\pi}}\varphi'(\xi)^{1/2}.
            \end{equation*}
        \end{proof}
    \end{proposition}
    We are not sure if the same result is true for other $p$. In light of this it is natural to study the extremizers of the point evaluation functionals corresponding to real points. For $\xi \in \mathbb{R}$ we let $C(p,E,\xi)$ be the best constant, such that
    \begin{equation*}
        \left\lvert \frac{f(\xi)}{E(\xi)} \right\rvert \leq C(p,E,\xi)\norm{f/E}_{p}, \; f \in \mathcal{H}^{p}(E).
    \end{equation*}
    For $p = 2$ extremal functions are, of course, given by the kernels functions. For $0 < p < \infty$ a normal families argument gives the existence of a function $f_{p, E, \xi} \in \mathcal{H}^{p}(E)$ satisfying
    \begin{equation}\label{eq:extremal_pe}
        f_{p, E, \xi}(\xi) = \lvert E(\xi) \rvert C(p,E,\xi), \text{ and } \norm{f_{p, E, \xi}/E}_{p} = 1.
    \end{equation}
    It is clear that
    \begin{equation*}
        C(p, E) = \sup_{\xi \in \mathbb{R}} \, C(p,E,\xi).
    \end{equation*}
    Most of the basic properties of the extremal functions in the Paley-Wiener case extend to the de Branges setting without extra difficulty. We shall state the results, but only give sketches of proofs if they are the same as in the Paley-Wiener case. The reader interested in more details can see Section 3 in \cite{point_eval_in_PW}.

    \begin{proposition}
        Let $f = f_{p,E,\xi}$ satisfy \eqref{eq:extremal_pe} (for some $p, \xi$, and $E$). Then $f$ is real entire.
        \begin{proof}
            Indeed $g = 2^{-1}(f+f^{\#})$ satisfies $g(\xi) = \lvert E(\xi)\rvert C(p,E,\xi)$ and if $f(x)$ is not real for some real $x$, then $\lvert g(x) \rvert = \lvert \text{Re}(f(x)) \rvert < \lvert f(x) \rvert$, contradicting that $f$ is extremal.
        \end{proof}
    \end{proposition}

    \begin{proposition}\label{prop:variational_int}
        Let $f = f_{p,E,\xi}$ satisfy \eqref{eq:extremal_pe} (for some $p, \xi$, and $E$). Suppose $r = r_{1}/r_{2}$ is a rational function with $\text{deg}(r_{2}) \geq \text{deg}(r_{1})$, such $r(\xi) = 0$ and $rf$ is entire. Then
        \begin{equation*}
            \int_{-\infty}^{\infty} \frac{r(x) \lvert f(x) \rvert^{p}}{\lvert E(x) \rvert^{p}}dx = 0.
        \end{equation*}
        \begin{proof}
            The point is that the function
            \begin{equation*}
                F(\epsilon) = \int_{-\infty}^{\infty} \frac{\lvert f(x) + \epsilon f(x)r(x) \rvert^{p}}{\lvert E(x)\rvert^{p}} dx,
            \end{equation*}
            has a minimum at $0$ since $f$ is extremal and $r(\xi) = 0$. Thus if $F$ is differentiable $F'(0) = 0$. One can justify differentiation under the integral which implies the result since
            \begin{equation*}
                \partial_{\epsilon} \lvert f(x) + \epsilon f(x)r(x) \rvert^{p} = p \lvert f(x) + \epsilon f(x)r(x) \rvert^{p-2} \lvert f(x) \rvert^{2} \text{Re}(r(x)\overline{(1+\epsilon r(x)}),
            \end{equation*}
            whenever $f(x) + \epsilon f(x)r(x) \neq 0$. Indeed, setting $\epsilon = 0$ gives
            \begin{equation*}
                0 = F'(0) = \int_{-\infty}^{\infty} \frac{\text{Re}(r(x)) \lvert f(x) \rvert^{p}}{\lvert E(x) \rvert^{p}}dx.
            \end{equation*}
            The result is thus proved for real $r$. For general $r$ applying the result with $r$ replaced with $ir$ proves the theorem.
        \end{proof}
    \end{proposition}

    \begin{corollary}
        Let $f=f_{p,E,\xi}$ satisfy \eqref{eq:extremal_pe}. Then $f$ has only real simple zeros.
        \begin{proof}
            Suppose, for a contradiction, that $w$ is a non-real zero of $f$. Then since $f$ is real entire $\overline{w}$ is also a zero of $f$, thus for $\epsilon > 0$,
            \begin{equation*}
                g(z) = \frac{(z-w)(z-\overline{w})-\epsilon(z-\xi)^{2}}{(z-w)(z-\overline{w})}f(z) \in \mathcal{H}^{p}(E),
            \end{equation*}
            and $g(\xi) = f(\xi)$. Since for real $x \neq \xi$,
            \begin{equation*}
                \lvert g(x) \rvert = \frac{\lvert \lvert x - w \rvert^{2} -\epsilon(x-\xi)^{2} \rvert}{\lvert x-w \rvert^{2}}\lvert f(x) \rvert < \lvert f(x) \rvert,
            \end{equation*}
            for sufficiently small $\epsilon > 0$, we obtain the desired contradiction. We turn to proving that the zeros are simple. Indeed, suppose $\lambda \in \mathbb{R}$ is a zero of or order greater than or equal to $2$. Applying Proposition \ref{prop:variational_int} with $r(z) = (z-\xi)^{2}/(z-\lambda)^{2}$ gives
        \begin{equation*}
            \int_{-\infty}^{\infty} \frac{(x-\xi)^{2} \lvert f(x) \rvert^{p}}{(x-\lambda)^{2}\lvert E(x) \rvert^{p}}dx = 0,
        \end{equation*}
        which is a contradiction since the integrand is nonnegative.
        \end{proof}
    \end{corollary}

    We also have the following orthogonality relationship between zeros of an extremal function.

    \begin{proposition}
        Let $f=f_{p, E, \xi}$ satisfy\eqref{eq:extremal_pe} and $\lambda_{1}, \lambda_{2}$ be zeros of $f$. Then
        \begin{equation*}
            \int_{-\infty}^{\infty} \frac{(x-\xi)^{2} \lvert f(x) \rvert^{p}}{(x-\lambda_{1})(x-\lambda_{2})\lvert E(x) \rvert^{p}}dx = 0,
        \end{equation*}
        \begin{proof}
            Applying Proposition \ref{prop:variational_int} with $r(x) = (x-\lambda_{1})^{-1}(x-\lambda_{2})^{-1}(x-\xi)^{2}$ proves the statement.
        \end{proof}
    \end{proposition}

    \subsection{$1 \leq p < \infty$}

    Now let us turn our attention to the convex range $p \geq 1$. In this case it is not difficult to show that there is a unique extremal function.
    \begin{proposition}
        Let $1 \leq p < \infty$. There exists a unique solution of \eqref{eq:extremal_pe}.
        \begin{proof}
            Existence has already been established. Suppose $f$ and $g$ are two solutions of \eqref{eq:extremal_pe}. Our goal is to show that $f = g$. For $0 \leq t \leq 1$ let
            \begin{equation*}
                F_{t}(z) = (1-t)f(z) + tg(z).
            \end{equation*}
            Then $F_{t}(\xi)/\lvert E(\xi) \rvert = C(p,E,\xi)$. It follows that $\norm{F_{t}/E}_{p} \geq 1$. Also, by the triangle inequality
            \begin{equation*}
                \norm{F_{t}/E}_{p} \leq (1-t)\norm{f/E}_{p} + t\norm{g/E}_{p} = 1.
            \end{equation*}
            Thus $\norm{F_{t}/E}_{p} = 1$. Since we have equality in the triangle inequality it follows that $f$ and $g$ are parallel and hence equal since they agree at $\xi$.
        \end{proof}
    \end{proposition}
    Uniqueness implies that the extremal function is of mean-type $0$. The mean type of a holomorphic function is the upper half-plane is defined to be
    \begin{equation*}
        \limsup_{y \to \infty} \frac{\log \lvert f(iy) \rvert}{y}.
    \end{equation*}
    Mean type in the lower half-plane is defined analogously. For $PW^{p}_{\pi}$ it was shown in \cite{point_eval_in_PW} that the extremal function is of exponential type $\pi$. This is equivalent to at least one of the functions $e^{i\pi z}f$ and $e^{i\pi z}f^{\#}$ having mean-type $0$ in the upper half-plane. We show that, for $p \geq 1$, this remains true in the de Branges space setting.

    \begin{proposition}
        Let $1 \leq p < \infty$ and $f$ be extremal for \eqref{eq:extremal_pe}. Then $f/E$ and $f^{\#}/E$ have mean-type $0$.
        \begin{proof}
            Since $f=f^{\#}$ it suffices to prove it for one of the functions. Suppose, for a contradiction, $f/E$ has strictly negative mean-type. Then for sufficiently small $\epsilon > 0$ the function $e^{-i\epsilon z}f$ belongs to $\mathcal{H}^{p}(E)$ and is also extremal. Since it is not a multiple of $f$ this is a contradiction.
        \end{proof}
    \end{proposition}

\section{Separation of Zeros}

    In \cite{point_eval_in_PW} it was shown that for $p \geq 1/2$ the zeros of the extremal functions in $PW^{p}_{\pi}$ are uniformly separated, that is, there exists a constant $\delta > 0$, such that
    \begin{equation*}
        \lvert \lambda_{1} - \lambda_{2} \rvert > \delta,
    \end{equation*}
    for all $\lambda_{1}, \lambda_{2}$, such that $f(\lambda_{1})=f(\lambda_{2})=0$. Our goal in this section is to apply our extension of Hörmander's inequality to extend this result to the de Branges setting. We continue to assume that our Hermite-Biehler function has no real zeros. For $p = 2$ the zeros of the extremal function are uniformly separated if and only if $\varphi' \in L^{\infty}$, thus we consider only de Branges spaces with this property. Unfortunately, we are only able to handle the case $p \geq 1$.

    \begin{theorem}
        Let $p \geq 1$, $\varphi' \in L^{\infty}$, and $f = f_{p,E,\xi}$ satisfy \eqref{eq:extremal_pe}. Then the zeros of $f$ are uniformly separated.
    \end{theorem}

    Before we begin with the proof let us make some comments. It is obvious that $\cos(x)$ has uniformly separated zeros and that
    \begin{equation*}
        \lvert \cos(x) \vert \geq 1/2 \text{, for } x \in (\pi k - \pi/3,\pi k + \pi/3), k \in \mathbb{Z}.    
    \end{equation*}
    The key part of the second statement is that the absolute value cosine is bounded below by a uniform constant, $1/2$, on intervals containing a maxima with length bounded from below by uniform constant, $2\pi/3$. The next two lemmas are analogous statements for $A_{\alpha}$ when the derivative of the phase function is bounded.

    \begin{lemma}
        Let $\varphi' \in L^{\infty}$. Then the zeros of $A_{\alpha}$ are uniformly separated with constant independent of $\alpha$.
        \begin{proof}
            The zeros of $A_{\alpha}$ are given by the set
            \begin{equation*}
                X(\alpha) = \left\{ x \in \mathbb{R} : \varphi(x) = 2\alpha + \pi \mod 2\pi \right\}.
            \end{equation*}
            Let us label the zeros $\varphi(x_{n}) = 2\alpha + \pi + 2\pi n$, $n \in \mathbb{Z}$. The phase function has strictly positive derivative and hence it is invertible with differentiable inverse. Let $x_{n}, x_{m}$ be two zeros of $A_{\alpha}$. An application of the mean-value theorem gives
            \begin{equation*}
            \begin{split}
                \lvert x_{m} - x_{n} \rvert = \lvert \varphi^{-1}(2\alpha + \pi + 2\pi m) - \varphi^{-1}(2\alpha + \pi + 2\pi n)\rvert \\ = \lvert \partial_{x} \varphi^{-1}(y_{n,m}) \rvert \lvert 2\alpha + \pi + 2\pi m - 2\alpha + \pi + 2\pi n \rvert \geq \frac{2\pi}{\norm{\varphi'} _{\infty}} \lvert n-m \rvert,
            \end{split}
            \end{equation*}
            where $y_{n,m}$ is some point between $x_{n}$ and $x_{m}$ given by the mean-value theorem.
        \end{proof}
    \end{lemma}

    \begin{remark}
        The previous lemma is Theorem \ref{thm:sep_zeros} for $p = 2$.
    \end{remark}

    \begin{lemma}\label{lemma:sep_from_0}
        Let $\varphi' \in L^{\infty}$, $\xi \in \mathbb{R}$ and $\alpha \in [0,\pi)$. There exists a constant $\delta > 0$ depending only on $\norm{\varphi'}_{\infty}$, such that whenever $\lvert A_{\alpha}(\xi) \rvert = \lvert E(\xi) \rvert$ there exists an interval $I = (a,b)$ containing $\xi$ with 
        \begin{equation*}
            \begin{cases}
                \xi - a > \delta, \\
                b-\xi > \delta
            \end{cases}
            \text{and } \; \; \left\lvert \frac{A_{\alpha}(x)}{E(x)}\right\rvert^{2} \geq 1/2, \; x \in I.
        \end{equation*}
        \begin{proof}
            The equality $\lvert A_{\alpha}(\xi) \rvert = \lvert E(\xi) \rvert$ is equivalent to $B_{\alpha}(\xi) = 0$, which, as in the previous lemma, is seen to be equivalent to
            \begin{equation*}
                \varphi(\xi) = 2\alpha \mod 2\pi.
            \end{equation*}
            Also,
            \begin{equation*}
                \left\lvert \frac{A_{\alpha}(x)}{E(x)}\right\rvert^{2} \geq 1/2,
            \end{equation*}
            is equivalent to
            \begin{equation*}
                \left\lvert A_{\alpha}(x)\right\rvert \geq \lvert B_{\alpha}(x) \rvert.
            \end{equation*}
            Since $\varphi(x) = -2\text{arg}(E)$ and $A_{\alpha}$, $B_{\alpha}$ are the real and imaginary parts of $E_{\alpha}$ respectively this is seen to be equivalent to
            \begin{equation*}
                \lvert \sin(\alpha-2^{-1}\varphi(x)) \rvert \leq \lvert \cos(\alpha-2^{-1}\varphi(x)) \rvert. 
            \end{equation*}
            This is true if
            \begin{equation*}
                \frac{-\pi}{4} \leq \frac{\varphi(x)}{2} - \alpha \leq \frac{\pi}{4} \mod \pi.
            \end{equation*}
            Let $\varphi(\xi) = 2\alpha + 2\pi k$, $k \in \mathbb{Z}$. Let $\xi_{l}$ be the point to the left of $\xi$, such that $\varphi(\xi_{l}) = -\frac{\pi}{2} + 2\alpha + 2\pi k$. We must show that $\lvert \xi - \xi_{l} \rvert$ is bounded below uniformly in $\alpha$ and $k$. By the mean-value theorem we have
            \begin{equation*}
                \lvert \xi - \xi_{l} \rvert = \lvert \varphi^{-1}(2\alpha + 2\pi k)- \varphi^{-1}(-\frac{\pi}{2} + 2\alpha + 2\pi k)\rvert \geq \frac{1}{\norm{\varphi'}_{\infty}}\frac{\pi}{2}.
            \end{equation*}
            The same argument applies to the right and hence we may choose $\delta = \pi \left(2\norm{\varphi'}_{\infty}\right)^{-1} > 0$.
        \end{proof}
    \end{lemma}

    We are now ready for the proof of Theorem \ref{thm:sep_zeros}.
    
        \begin{proof}[Proof of Theorem \ref{thm:sep_zeros}]
            Recall that $p \geq 1$ and $f = f_{p,\xi, E}$. For notational simplicity we assume $\xi = 0$. Let us index the zeros of $f$ by
            \begin{equation*}
               ... < \lambda_{-2} < \lambda_{-1} < 0 \leq \lambda_{0} < \lambda_{1} < \lambda_{2} < ...
            \end{equation*}
            Let $\lambda_{n+1} > \lambda_{n}$ be two consecutive zeros of $f$, $I_{n} = [\lambda_{n}, \lambda_{n+1}]$. The proof when the zeros are negative requires only straight forward modifications, hence we confine ourselves to the case $\lambda_{n+1} > \lambda_{n} > 0$. Applying Proposition \ref{prop:variational_int} we obtain
            \begin{equation}\label{eq:pf_thm_sep_zeros}
                \int_{I_{n}} \frac{x^{2}\lvert f(x) \rvert^{p}}{(x-\lambda_{n})(\lambda_{n+1}-x)\lvert E(x) \rvert^{p}}dx = \int_{\mathbb{R} \setminus I_{n}} \frac{x^{2}\lvert f(x) \rvert^{p}}{\lvert (x-\lambda_{n})(\lambda_{n+1}-x) \rvert \lvert E(x) \rvert^{p}}dx.
            \end{equation}
            Let
            \begin{equation*}
                \psi_{n}(z) = \frac{z^{2}f(z)}{(z-\lambda_{n})(\lambda_{n+1}-z)} \in \mathcal{H}^{p}(E).
            \end{equation*}
            The left hand-side of \eqref{eq:pf_thm_sep_zeros} can be estimated as follows
            \begin{equation*}
                \int_{I_{n}} \left\lvert \frac{x^{2}}{(x-\lambda_{n})(\lambda_{n+1}-x)}\right\rvert^{1-p}\left\lvert \frac{\psi_{n}(x)}{E(x)} \right\rvert^{p}dx \leq \norm{\psi_{n}/E}_{\infty}^{p} \int_{I_{n}} \left\lvert \frac{x^{2}}{(x-\lambda_{n})(\lambda_{n+1}-x)}\right\rvert^{1-p}dx.
            \end{equation*}
            Now we shall estimate the right hand-side of \eqref{eq:pf_thm_sep_zeros}. The idea is to localise around a maximum of the function $\psi_{n}/E$ and use our generalization of Hörmander's inequality to obtain a a lower bound. We now present the details.
            
            Let $\delta > 0$ be such that whenever $\lvert A_{\alpha}(\xi) \rvert = \lvert E(\xi) \rvert$ there exists an interval $(a,b)$ containing $\xi$ with 
            \begin{equation*}
                \begin{cases}
                    \xi - a > \delta, \\
                    b-\xi > \delta
                \end{cases}
            \text{and } \; \; \left\lvert \frac{A_{\alpha}(x)}{E(x)}\right\rvert^{2} \geq 1/2, \; x \in (a,b).
            \end{equation*}
            The existence of such a $\delta$ is guaranteed by Lemma \ref{lemma:sep_from_0}. We may assume that $\lvert I_{n} \rvert < \delta/3$. Since, $\psi_{n} \in \mathcal{H}^{p}(E)$ there exists $x_{0} \in \mathbb{R}$, such that
            \begin{equation*}
                \frac{\lvert \psi_{n}(x_{0}) \rvert}{\lvert E(x_{0}) \rvert} = \norm{\psi_{n}/E}_{\infty}.
            \end{equation*}
            Let $\alpha$ be such that $e^{i\alpha}E(x_{0}) = \lvert E(x_{0}) \rvert$. Then there exists an interval $J_{\delta}^{n}$, such that
            \begin{equation*}
                \begin{cases}
                    \lvert J_{\delta}^{n} \rvert \geq \delta/6, \\
                    \text{dist}(J_{\delta}^{n}, I_{n}) \geq \delta/6.
                \end{cases}
                \text{ and } \lvert \psi_{n}(x_{0}) \rvert \geq \norm{\psi_{n}/E}_{\infty}A_{\alpha}(x) \geq 2^{-1/2}\norm{\psi_{n}/E}_{\infty}\lvert E(x) \rvert, x \in J_{\delta}^{n}.
            \end{equation*}
            Restricting the integration in the integral in the right hand-side of \eqref{eq:pf_thm_sep_zeros} to the interval $J_{\delta}^{n}$ gives
            \begin{equation*}
                \int_{\mathbb{R} \setminus I_{n}} \frac{x^{2}\lvert f(x) \rvert^{p}}{\lvert (x-\lambda_{n})(\lambda_{n+1}-x)  E(x) \rvert^{p}}dx \geq 2^{-p/2}\norm{\psi_{n}/E}_{\infty}^{p}\int_{J_{\delta}^{n}} \left\lvert \frac{x^{2}}{(x-\lambda_{n})(\lambda_{n+1}-x)}\right\rvert^{1-p}dx.
            \end{equation*}
            Combining the estimates for the left and right hand-side we obtain
            \begin{equation}\label{eq:pf_sep_0_integ_ineq}
                \int_{I_{n}} \left\lvert \frac{x^{2}}{(x-\lambda_{n})(\lambda_{n+1}-x)}\right\rvert^{1-p}dx \geq 2^{-p/2}\int_{J_{\delta}^{n}} \left\lvert \frac{x^{2}}{(x-\lambda_{n})(\lambda_{n+1}-x)}\right\rvert^{1-p}dx.
            \end{equation}
            Our goal is now to show that this inequality implies that the length of $I_{n}$ is bounded below by a uniform constant. We shall estimate the left hand-side with
            \begin{equation}\label{eq:pf_sep_0_right_est}
                \int_{I_{n}} \left\lvert \frac{x^{2}}{(x-\lambda_{n})(\lambda_{n+1}-x)}\right\rvert^{1-p}dx \leq \lambda_{n}^{2(1-p)} \lvert I_{n} \rvert^{2p-1}B(p,p),
            \end{equation}
            where $B$ is the beta function, see Appendix A in \cite{MR3243734}. We now begin to estimate the right hand-side. The change of variable $t = \lvert I_{n} \rvert^{-1} (x-\lambda_{n})$ gives
            \begin{equation*}
                \begin{split}
                    \int_{J_{\delta}^{n}} \left\lvert \frac{x^{2}}{(x-\lambda_{n})(\lambda_{n+1}-x)}\right\rvert^{1-p}dx = \lvert I_{n} \rvert^{2p-1} \int_{\lvert I_{n} \rvert^{-1}(J_{\delta}-\lambda_{n})} \lvert \lambda_{n+1} t + (1-t)\lambda_{n} \rvert^{2(1-p)} \lvert t \rvert^{p-1} \lvert 1-t \rvert^{p-1}dt
                    \\ = \lvert I_{n} \rvert^{2p-1} \int_{\lvert I_{n} \rvert^{-1}(J_{\delta}-\lambda_{n})} \left(\lambda_{n+1} \left\lvert \frac{t}{1-t} \right\rvert^{1/2} - \left\lvert \frac{(1-t)}{t} \right\rvert^{1/2}\lambda_{n} \right)^{2(1-p)}dt.
                \end{split}
            \end{equation*}
            To get the signs right is it useful to use that in the old coordinates $t/(t-1) = (x-\lambda_{n})/(x-\lambda_{n+1})$. Since we have assumed $\lvert I_{n} \rvert \leq \delta/3$ the quantity $\lvert t/(1-t) \rvert$ is comparable to a constant, in fact
            \begin{equation*}
                \frac{1}{3} \leq \left\lvert \frac{t}{t-1} \right\rvert \leq 3, \; t \in \lvert I_{n} \rvert^{-1}(J_{\delta}-\lambda_{n}).
            \end{equation*}
            Thus for $p \geq 1$ the triangle inequality gives
            \begin{equation*}
                \int_{\lvert I_{n} \rvert^{-1}(J_{\delta}-\lambda_{n})} \left(\lambda_{n+1} \left\lvert \frac{t}{1-t} \right\rvert^{1/2} - \left\lvert \frac{(1-t)}{t} \right\rvert^{1/2}\lambda_{n} \right)^{2(1-p)}dt \geq 3^{(1-p)}\lvert I_{n} \rvert^{-1} \lvert J_{\delta}^{n} \rvert 4^{(1-p)}\lambda_{n+1}^{2(1-p)}.
            \end{equation*}
            Combining this with the upper estimate shows that
            \begin{equation*}
                \lvert I_{n} \rvert \geq B(p,p)^{-1}3^{1-p}4^{1-p}\frac{\delta}{6}2^{-p/2}\frac{\lambda_{n+1}^{2(1-p)}}{\lambda_{n}^{2(1-p)}}.
            \end{equation*}
            Since the right hand-side is bounded below independent of $n$ we are done.
        \end{proof}

        The reason we are not able to prove that the zeros are uniformly separated when $0 < p < 1$ is because we do not know the location of the maximum of $\psi_{n}/E$ and hence the location of the interval $J_{\delta}^{n}$. If one were able to prove that the maximum of $\psi_{n}/E$ is close to the interval $I_{n}$ (in some suitable sense, adjacent would be sufficient) then one could extend the argument used in the proof to all $p$.

\section{Model Spaces}\label{sec:model_spaces}

In this section we briefly explain how our results in de Branges spaces can be interpreted as results for model spaces generated by meromorphic inner functions. An inner function, $\Theta$, is a bounded holomorphic function defined on the upper half-plane satisfying $\lvert \Theta(x) \rvert = 1$, for almost every $x \in \mathbb{R}$. For any $0 < p \leq \infty$ and inner function $\Theta$ we define the model space
\begin{equation*}
    K_{\Theta}^{p} = \left\{ f \in H^{p} \cap L^{1}_{\text{loc}}(\mathbb{R}) : \Theta \overline{f} \in H^{p} \right\}.
\end{equation*}
If $p \geq 1$ the condition that the functions should be $L^{1}_{\text{loc}}$ can be dropped. A meromorphic inner function is a meromorphic function in the plane, which is an inner function in the upper half-plane. As we have seen any $E \in HB$ defines a natural meromorphic inner function via the formula $\Theta_{E} = E^{\#}/E$. It is known that the map
\begin{equation*}
    E^{-1} : \mathcal{H}^{p}(E) \to K_{\Theta_{E}}^{p},
\end{equation*}
is a surjective isometry, see, for example, Proposition 2.8. in \cite{MR2215727}. Also, for any meromorphic inner-function $\Theta$ there exists a (not unique) $E \in HB$, such that $\Theta = \Theta_{E}$, see \cite{MR0229011}. Thus all of the results imply analogous results in $K_{\Theta}^{p}$. It is worth mentioning that $\norm{\Theta_{E}'}_{\infty} = \norm{\varphi'}_{\infty}$.

\bibliographystyle{abbrv}
\bibliography{citations}

\end{document}